\documentclass[11pt,reqno]{amsart}
\usepackage{amsaddr}
\usepackage[top=30truemm,bottom=30truemm,left=25truemm,right=25truemm]{geometry}
\usepackage[english]{babel}

\usepackage[dvips]{graphicx}
\usepackage{color}
\usepackage{amsthm}
\usepackage{amsmath}
\usepackage{amscd}
\usepackage{comment}

\newtheorem{theorem}{Theorem}[section]

\newtheorem{lemma}[theorem]{Lemma}
\newtheorem{proposition}[theorem]{Proposition}

\usepackage{amssymb}\theoremstyle{definition}

\newtheorem{example}[theorem]{Example}

\newtheorem{remark}[theorem]{Remark}

\makeatletter
\@namedef{subjclassname@2010}{
 \textup{2010} Mathematics Subject Classification}
\makeatother

\begin{document}
  \title[Periodic structure for piecewise contracting map]{Periodic structures for nonlinear piecewise contracting maps}

  \author[F. Nakamura]{Fumihiko Nakamura}
  \address[F. Nakamura]{Kitami Institute of Technology, Kitami, 090-8507, Japan}
  \email[F.Nakamura]{nfumihiko@mail.kitami-it.ac.jp}

    \date{}


   \begin{abstract}
In this paper, we first show that any nonlinear monotonic increasing contracting maps with one discontinuous point on a unit interval which has an unique periodic point with period $n$ conjugates to a piecewise linear contracting map which has periodic point with same period. Second, we consider one parameter family of monotonic increasing contracting maps, and show that the family has the periodic structure called Arnold tongue for the parameter which is associated with the Farey series. This implies that there exist a parameter set with a positive Lebesgue measure such that the map has a periodic point with an arbitrary period. Moreover, the parameter set with period $(m+n)$ exists between the parameter set with period $m$ and $n$.
   \end{abstract}

  \maketitle

\section{Introduction}

The piecewise linear contracting map on a unit interval defined by
\begin{eqnarray}\label{NS}
S_{\alpha,\beta}(x)=\alpha x+\beta \ ({\rm mod}  \ 1),\ \ \ \text{with}\ \ \ 0<\alpha,\beta<1
\end{eqnarray}
is known as Nagumom-Sato model \cite{Nagumo} or Keener map\cite{Keener}, which describes a dynamics of a single neuron introduced from the Caianiello's model\cite{Caianiello}. The map \eqref{NS} has a unique discontinuous point when $\alpha+\beta>1$, and it is known that the map has a unique periodic point for almost every parameter $\alpha$ and $\beta$ and a period with an arbitrary numbers can be found by choosing appropriate parameters \cite{N1}. The parameter region divided by a set of the parameters for which the map has periodic point with each period is known as {\it Arnold tongue} or {\it Farey structure} which shows a layered structure based on a Farey series. More precisely, there is a periodic region with period $n_1+n_2$ between the regions with period $n_1$ and $n_2$ if $|l_1n_2-l_2n_1|=1$ for irreducible fractions $l_1/n_1$ and $l_2/n_2$ in its parameter space.

The Arnold tongue is observed in many researches, for instance \cite{Glass,Boyland,Swiatek} showed the structure for the standard circle maps which describes a cardiac oscillation model. Such structure for the piecewise linear model has already been studies in, for example \cite{Keener,Ding,Hata1,N1}, especially Keener's results cover the nonlinear models and show for almost all parameters the system has a rational rotation number so that it possesses a periodic point. Moreover, an irrational rotation number is achieved on the Cantor set in the parameter space.

In this paper, we focus on the {\it nonlinear} model with a discontinuous point and show that any nonlinear system with certain conditions conjugates the linear contracting map \eqref{NS}. Furthermore, we prove that the family of parametrized nonlinear systems has the layered structure based on Farey series. In \cite{N1}, although similar layered structures were numerically observed for non-linear models (e.g. $T(x)=\alpha x^2+\beta$ (mod 1), $T(x)=\alpha \sqrt{x}+\beta$ (mod 1), etc), the rigorous proof for the maps could not be accomplished because of the difficulty for the non-linearity. 

The organization of this paper is as follows. In section \ref{S2}, we prepare some notations and previous results for the linear model. In section \ref{S3}, we show that almost any nonlinear monotonic increasing maps with one discontinuous point conjugates to the linear contracting maps \eqref{NS} for some parameter by constructing a concrete homeomorphism. This implies that any two nonlinear maps with a periodic point with same period can be conjugated each other under certain conditions. Note that the nonlinear map does not require the contracting property. In section \ref{S4}, we consider one parameter family of nonlinear monotonic increasing contracting maps, and give the proof of an existence of Farey structure on the parameter space. This implies that we can find the map with an arbitrary period by choosing an appropriate parameter for a family of nonlinear maps, and justified a part of Farey structure observed in \cite{N1}.

\section{Preliminary}\label{S2}

In this section, we introduce the facts for linear contracting maps \eqref{NS} with a periodic point (See \cite{N1,N2} for the details). Let $Pr(n):=\{l<n\ |\ GCD(n,l)=1\}$ for each $n\in \mathbb{N}$. In this paper, if we write $(n,l)$, then $n$ and $l$ always satisfy $n\in\mathbb{N}_{\geq 2}$ and $l\in Pr(n)$. We define two functions $B^U_{n,l}(\alpha)$ and $B^L_{n,l}(\alpha)$ and sets $\{D_{n,l}\}_{(n,l)}$ as follows;
\begin{eqnarray}
B^U_{n,l}(\alpha)&=&(1-\alpha )\left(\frac{1}{1-\alpha^n}\sum^{n-1}_{m=1}k_m\alpha^m +1\right),\label{BU}\\
B^L_{n,l}(\alpha)&=&(1-\alpha )\left(\frac{1}{1-\alpha^n}\sum^{n-1}_{m=1}k_m\alpha^m +1-\frac{\alpha^{n-1}-\alpha^n}{1-\alpha^n}\right),\label{BL}
\end{eqnarray}
and
\begin{eqnarray}\label{ma}
D_{n,l}=\{(\alpha,\beta) \in\ (0,1)^2\ |\ B^L_{n,l}(\alpha) \leq \beta <B^U_{n,l}(\alpha)\},
\end{eqnarray}
where
\begin{eqnarray}\label{gs}
k_m :=\left[ \frac{(m+1)l}{n} \right] -\left[ \frac{ml}{n}\right] \ \ \ {\rm for}\ \ \ m\in \mathbb{Z},
\end{eqnarray}
where $[x]$ is the integer part of $x$.  In \cite{N1}, we showed that if the parameter $\alpha$ and $\beta$ are chosen in the set $D_{n,l}$, then the map $S_{\alpha,\beta}$ has periodic point with period $n$. It is clear that the set $D_{n,l}$ has a positive Lebesgue measure. Moreover, the set $D_{n+n',l+l'}$ exists between $D_{n,l}$ and $D_{n',l'}$ if $|n'l-nl'|=1$. In this way, we can obtain explicit formula for the parameter region which the map $S_{\alpha,\beta}$ has a periodic point for the linear maps, and show that these regions display a layered structure associated with Farey series. Moreover, when $(\alpha,\beta)\in D_{n,l}$, the $n$ periodic points for the map $S_{\alpha,\beta}$ is given by
\begin{eqnarray}
{\rm Per}_n(S_{\alpha,\beta})=\left\{\frac{\beta}{1-\alpha}-A_i(\alpha)\ \middle|\ i=0,\cdots,n-1 \right\},
\end{eqnarray}
where, for $i=0,1,\cdots, n-1$,
\begin{eqnarray}\label{Ai}
A_i(\alpha)=\frac{1}{1-\alpha ^n}\left( \sum_{m=0}^{i-1}k_m\alpha^{i-m-1}+\sum_{m=i}^{n-1}k_m \alpha ^{n+i-m-1}\right).
\end{eqnarray}

We call the sequence $\{k_i\}_{i\in\mathbb{Z}}$  defined by (\ref{gs}) a {\bf rational characteristic sequence} with respect to $(n,l)$. The following propositions give the properties of rational characteristic sequences, and plays an important role to prove Theorem \ref{mainthm2}.
  
\begin{proposition}\label{PoGS}(\cite{N1}, Proposition 2.2)
Let $\{k_m\}_{m\in \mathbb{Z}}$ be a rational characteristic sequence with respect to $(n,l)$. We then have the following properties.
\begin{itemize}
 \item[{\rm(i)}]$k_{m\pm n}=k_m \ \ \ {\rm for}\ \ \ m\in \mathbb{Z}$,
 \item[{\rm(ii)}]$k_{n-1-m}=k_m \ \ \ {\rm for}\ \ \ m\in \mathbb{Z},\ m\notin n\mathbb{Z},n\mathbb{Z}-1$,
 \item[{\rm(iii)}] $k_{m-\hat{l}}=k_m  \ \ \ {\rm for}\ \ \ m\in \mathbb{Z},\ m\notin n\mathbb{Z},n\mathbb{Z}-1$,
\end{itemize}
where $\hat{l}:=\min\{ t\in \mathbb{N}\ |\ tl=1\  ({\rm mod}\  n)\}$. Note that $k_0=0$ and $k_{n-1}=1$ always hold obviously.
\end{proposition}

\begin{proposition}\label{impGS}(\cite{N1}, Proposition 2.3)
Let $\{k_m\}_{m\in\mathbb{Z}}$ be a rational characteristic sequence with respect to $(n,l)$ and $\{k'_m\}_{m\in\mathbb{Z}}$ be another rational characteristic sequence with respect to $(n',l')$. If $\frac{l}{n}<\frac{l'}{n'}$ and $nl'-n'l=1$, then the sequence $\{\hat{k}_m\}_{m\in\mathbb{Z}}$ defined by
\begin{eqnarray}
\hat{k}_{m}:=
\begin{cases}
k_m\ \ \ \ {\rm for}\ \ \ m=0,\cdots,n-1 & \\
k'_{m-n}\ \ \ \ {\rm for}\ \ \ m=n,\cdots,n+n'-1 &
\end{cases} \nonumber
\end{eqnarray}
and
\begin{eqnarray}
\hat{k}_{\bar{m}}:=\hat{k}_{m}\ \ \ {\rm if}\ \ \ \bar{m}=m+t(n+n')\ \ \ {\rm with}\ \ \ m=0,\cdots,n+n'-1\ \ {\rm and}\ \ t\in \mathbb{Z}\backslash \{0\}\nonumber
\end{eqnarray}
is the rational characteristic sequence with respect to $(n+n',l+l')$.
\end{proposition}

\begin{remark}\label{palindrome}
From (ii) in Proposition \ref{PoGS}, $(k_1k_2\cdots k_{n-2})=(k_{n-2}\cdots k_2k_1)$ holds when $\{k_i\}$ is a rational characteristic sequence. From Proposition \ref{impGS}, 
\begin{eqnarray}
(\hat{k}_1\cdots\hat{k}_{n+n'-2})&=&(k_1\cdots k_{n-1}k'_0\cdots k'_{n'-2})\nonumber\\
&=&(k_1\cdots k_{n-2} 1 0 k'_1\cdots k'_{n'-2})\nonumber
\end{eqnarray}
since $\{\hat{k}_i\}$, $\{k_i\}$ and $\{k'_i\}$ are rational characteristic sequences with respect to $(n+n',l+l')$, $(n,l)$ and $(n',l')$, using (ii) in Proposition \ref{PoGS} again, we have
\begin{eqnarray}
(\hat{k}_1\cdots\hat{k}_{n+n'-2})&=&(k'_{n'-2}\cdots k'_{1} 0 1 k_{n-2}\cdots k_{1})\nonumber\\
&=&(k'_1\cdots k'_{n'-2} 0 1 k_1\cdots k_{n-2})\nonumber
\end{eqnarray}
These calculations are used many times in the proof of Theorem \ref{mainthm2}.

\end{remark}

Next, considering the pre-images of $0$ or the discontinuous point is useful to analyse the contracting maps. Indeed, we can write the pre-image of $0$ for the map $S_{\alpha,\beta}$ explicitly as follows. We also consider the pre-image for the nonlinear map in the section 3.
\begin{proposition}{(\cite{N2}, Proposition 4)}\label{keypro3}
Assume that $(\alpha,\beta)\in D_{n,l}$, then
\begin{eqnarray}\label{ioz}
S_{\alpha,\beta}^{-i}(0)=\sum_{m=1}^i\frac{k_{n-i+m-1}-\beta}{\alpha^m}\in [0,1],\ \ \ \ (i=1,\cdots,n-1),
\end{eqnarray}
where $\{k_m\}$ is a rational characteristic sequence with respect to $(n,l)$. \\Moreover, for $i=n$, $S_{\alpha, \beta}^{-n}(0)$ is not in $[0,1]$.
\end{proposition}

In the end of this section, we introduce the results in \cite{Hata1, Hata2} which tell us that if parameters $(\alpha,\beta)$ and $(\alpha',\beta')$ are chosen from same set $D_{n,l}$, then the maps $S_{\alpha,\beta}$ and $S_{\alpha',\beta'}$ are conjugate each other. More precisely, the following proposition holds. 

\begin{proposition}[\cite{Hata1}, Theorem 7.1]
The followings hold:
\begin{itemize}
\item[(i)] If $(\alpha,\beta), (\alpha,\beta)\in int(D_{n,l})$, $S_{\alpha,\beta}$ and $S_{\alpha',\beta'}$ are conjugate.
\item[(ii)] If $\beta=B_{n,l}^{U}(\alpha)$ and $\beta'=B_{n,l}^{U}(\alpha')$, then $S_{\alpha,\beta}$ and $S_{\alpha',\beta'}$ are conjugate.
\item[(iii)] If $\beta=B_{n,l}^{L}(\alpha)$ and $\beta'=B_{n,l}^{L}(\alpha')$, then $S_{\alpha,\beta}$ and $S_{\alpha',\beta'}$ are conjugate.
\end{itemize}
\end{proposition}

\section{Conjugacy with nonlinear piecewise monotonic increasing maps}\label{S3}

In this section, we show that nonlinear piecewise monotonic increasing maps satisfying some assumptions conjugate the linear contracting ones.

Let  $f:[0,1]\to[0,1]$ be a continuous map except with $x=c$ $(0<c<1)$, satisfying 
\begin{itemize}
\item[(A1)] $f(0)>f(1)$ (called {\it non-overlapping condition}),
\item[(A2)] $f(c-)=1$ and $f(c+)=0$,
\item[(A3)] if $x<y$, then $f(x)<f(y)$ for $x,y\in[0,c)$ or $x,y\in[c,1]$,
\item[(A4)] there exists an integer $n\geq 2$ such that the pre-image of zero $f^{-i}(0)$ is in $[0,1)$ for $i=0,1,\cdots,n-1$ and $f^{-n}(0)=\emptyset$ .
\end{itemize}

\begin{remark}
In the above setting, we note that the pre-image for any point is unique if it exists. Then, we often use $f^{-i}(x)$ as the point iterating $x$ by the inverse map $f^{-i}$.
\end{remark}
\begin{remark}
Clearly if $x\in(f(1),f(0))$, then $f^{-1}(x)=\emptyset$. Moreover, 
since $f^{-1}(0)=c$, above assumption (A4) can be written by a pre-image of discontinuous point $c$ such as  
\begin{itemize}
\item[(A4)'] there exists an integer $n\geq 2$ such that
 $f^{-i}(c)\in [0,f(1))\cup[f(0),1)$ for $i=0,1,\cdots,n-2$ and $f^{-(n-1)}(c)\in(f(1),f(0))$ .
\end{itemize}
This form is used in the proof of Theorem \ref{mainthm2}.
\end{remark}

The main result in this section is the next theorem which conclude any nonlinear system satisfying (A1)-(A4) conjugates some linear systems.

\begin{theorem}\label{mainthm1}
  Let $f:[0,1]\to[0,1]$ be a continuous map except with $x=c$ $(0<c<1)$, satisfying (A1)-(A4). Then $f$ conjugates $S_{\alpha,\beta}$ with $(\alpha,\beta)\in int(D_{n,l})$.
\end{theorem}

To prove the theorem, we first prepare the following lemmas.

\begin{lemma}\label{lem0}
Let $l$ be a number of elements of the set $\{i\ |\ f^{-i}(0)\in[c,1] \}$. Let $\{x_i\}_{i=1}^{n-l}$ and $\{y_i\}_{i=1}^{l}$ be $n$ points of pre-images of zero, $\{f^{-i}(0)\}_{i=0}^{n-1}$, such that,
$$
x_1 < x_2 < \cdots < x_{n-l} < y_1 < y_2 < \cdots < y_l.
$$
If $n>2l$, then the following orbit relations hold;
\begin{align}
f^{-1}(x_i) &= y_i & \text{for} \quad\quad i&=1,\cdots,l,\label{or1}\\
f^{-1}(x_i)&= \emptyset & \text{for} \quad\quad i&=l+1,\label{or2}\\
f^{-1}(x_i) &= x_{i-l} & \text{for} \quad\quad i&=l+2,\cdots,n-l,\label{or3}\\
f^{-1}(y_i) &= x_{n-2l+i} & \text{for} \quad\quad i&=1,\cdots,l.\label{or4}
\end{align}
If $n<2l$, then the following orbit relations hold;
\begin{align}
f^{-1}(x_i) &= y_i & \text{for}\quad\quad  i&=1,\cdots,l,\label{or5}\\
f^{-1}(y_i) &= y_{i+l} & \text{for}\quad\quad  i&=1,\cdots,2l-n-2,\label{or6}\\
f^{-1}(y_i) &= \emptyset & \text{for}\quad\quad  i&=2l-n-1,\label{or7}\\
f^{-1}(y_i) &= x_{n-2l+i} & \text{for}\quad\quad  i&=2l-n,\cdots,l.\label{or8}
\end{align}
\end{lemma}

\begin{proof}
In the case $n>2l$, since $x_1=0$, $f^{-1}(x_1)=y_1=c$ and $f^{-1}(y_l)=x_{n-l}$, the relations \eqref{or1} and \eqref{or4} are immediately hold because of the monotonicity of the map. Then, one of $x_i$, $i=l+1,l+2,\cdots,n-l$ is mapped to $\emptyset$, and the others are mapped to $x_k$, $k=2,\cdots,n-2l$ by $f^{-1}$. By the monotonicity, the relations \eqref{or2} and \eqref{or3} must hold. 

For the case $n<2l$, we can show similarly by substituting the role of $x_i$ and $y_i$.
\end{proof}
In order to help understanding these orbit relations, we show the example of our target map with $(n,l)=(9,2)$ in Figure \ref{fig1}.

\begin{figure}[tpb]
\begin{center}
\includegraphics[bb=0 0 300 300, width=7cm]{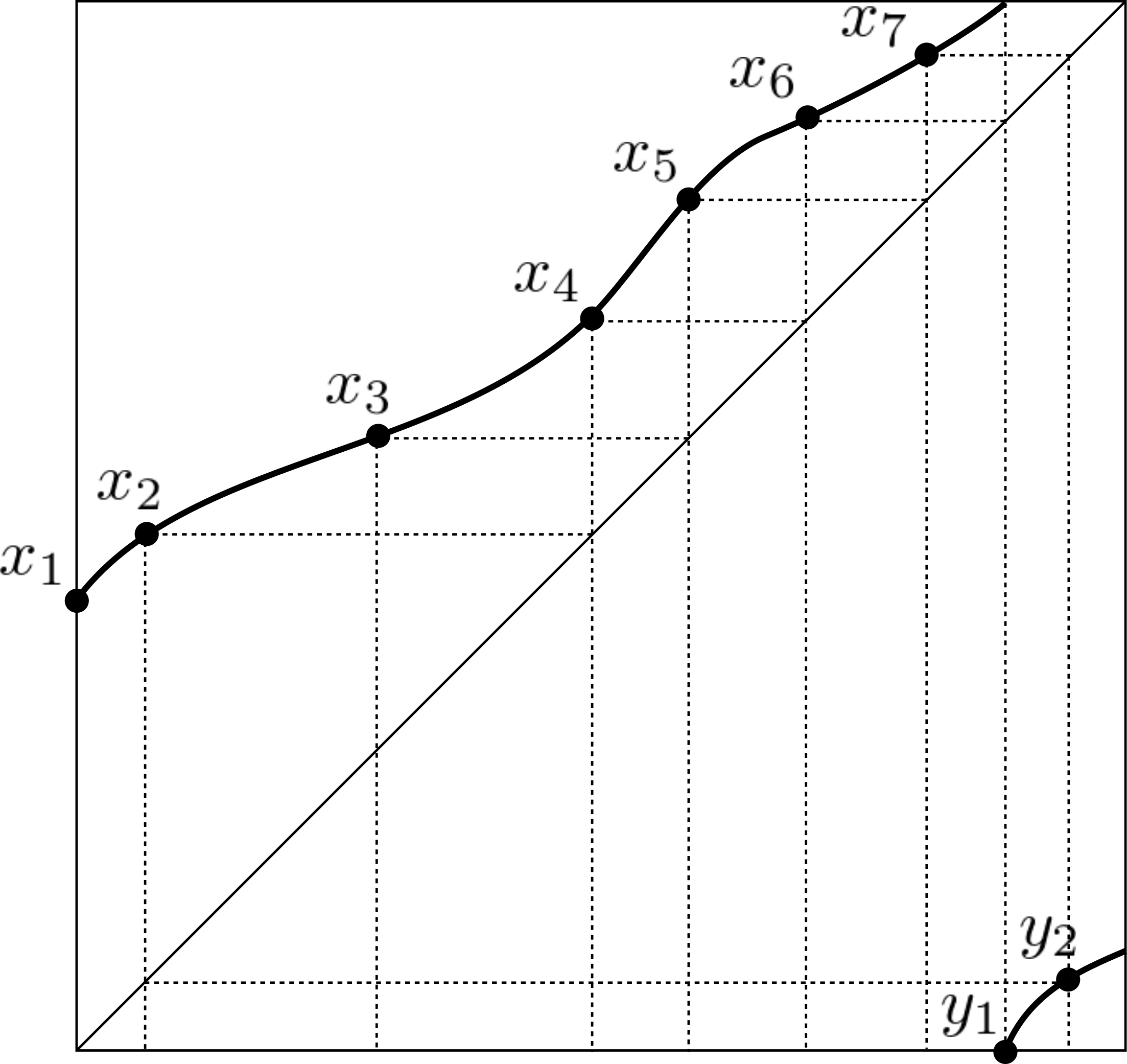}
\caption{Example of our target map with $(n,l)=(9,2)$.}
\label{fig1}
\end{center}
\end{figure}

\begin{lemma}\label{lem1}
Let $l$ be a number of elements of the set $\{i\ |\ f^{-i}(0)\in[c,1] \}$. Then $l$ and $n$ are relatively prime numbers.
\end{lemma}

\begin{proof}

When $n$ is a prime number, $n$ and $l$ are always relatively prime. Then we assume that $n=kn'$ and $l=kl'$ $(k>1, n'\geq 1, l'\geq 1)$ where $n'$ and $l'$ are relatively prime. Clearly $f^0(0)=0\in[0,c)$ and $f^{-1}(0)=c\in [c,1]$.

In the case $n>2l$, let $x_1^1<\cdots<x_{k}^1<x_1^2<\cdots<x_k^2<\cdots<x_1^{n'-l'}<\cdots<x_k^{n'-l'}<y_1^1<\cdots<y_k^1<\cdots<y_1^{l'}<\cdots<y_k^{l'}$ be points in $\{f^{-i}(0)\}_{i=0}^{n-1}$ such that $x_i^j\in[0,c)$ for $i=1,\cdots,k$ and $j=1,\cdots,n'-l'$, and $y_i^j\in[c,1]$ for $i=1,\cdots,k$ and $j=1,\cdots,l'$.

Consider pre-images of all $x_i^j$ and $y_i^j$ by $f$. First, by the orbit relation \eqref{or1} and \eqref{or4}, we have
\begin{eqnarray}\label{pieq1}
f^{-1}(x_i^j)=y_i^{j}\ \ \ \text{for}\ \ \  i=1,\cdots,k\ \  \text{and}\ \  j=1,\cdots,l',
\end{eqnarray}
\begin{eqnarray}\label{pieq2}
f^{-1}(y_i^j)=x_i^{n'-2l'+j}\ \ \ \text{for}\ \ \  i=1,\cdots,k\ \  \text{and}\ \  j=1,\cdots,l'.
\end{eqnarray}
and for reminding $x_2^{l'+1},\cdots, x_k^{l'+1}$ and $x_i^j$ for $i=1,\cdots,k$ and $j=l'+2,\cdots,n'-l'$, we have
\begin{eqnarray}\label{pieq4}
f^{-1}(x_i^j)=x_i^{j-l'}\quad \text{by \eqref{or3}}.
\end{eqnarray}
Moreover, for $x_1^{l'+1}$, we have
\begin{eqnarray}\label{pieq3}
x_1^{l'+1}=f^{-(n-1)}(0)\quad \text{by \eqref{or2}},
\end{eqnarray}
that is, there is no pre-image for  $x_1^{l'+1}$ by $f$. However, since the index $i$ of $x_i^j$ or $y_i^j$ is invariant by these rules \eqref{pieq1}-\eqref{pieq2} of iteration, the pre-image of $x_1^1(=0)$ traces on only $x_1^\cdot$ and $y_1^\cdot$, and reaches to $x_1^{l'+1}$. This contradicts to the assumption that all $x_i^j$ and $y_i^j$ are elements of pre-images of $0$, $\{f^{-i}(0)\ |\ i=0,1,\cdots,n-1\}$.

For the case $n<2l$, it can be shown similarly by substituting the role of $x_i^j$ and $y_i^j$.
\end{proof}

\begin{lemma}\label{lem2}
Let $l$ be a number given in previous Lemma \ref{lem1}. Let $\rho$ be a permutation which arranges $\{S_{\alpha,\beta}^{-i}(0)\}_{i=0}^{n-1}$ in increasing order, that is,
$$
S_{\alpha,\beta}^{-\rho(0)}(0)<S_{\alpha,\beta}^{-\rho(1)}(0)<\cdots<S_{\alpha,\beta}^{-\rho(n-1)}(0).
$$
where $(\alpha,\beta)\in int(D_{n,l})$.
Then, $\rho$ permutes $\{f^{-i}(0)\}_{i=0}^{n-1}$ in increasing order, that is,
$$
f^{-\rho(0)}(0)<f^{-\rho(1)}(0)<\cdots<f^{-\rho(n-1)}(0).
$$
\end{lemma}

\begin{proof}
From Lemma \ref{lem0}, if two maps $f$ and $g$ satisfying the assumption (A1) - (A4) with same number $n$ and $l$, then the orders of points $\{f^{-i}(0)\}_{i=0}^{n-1}$ and $\{g^{-i}(0)\}_{i=0}^{n-1}$ coincide. Especialy, the linear map $S_{\alpha,\beta}$ with $(\alpha,\beta)\in D_{n,l}$ also satisfies the assumption (A1) - (A4) from the facts in previous section. Therefore the orders of points $\{f^{-i}(0)\}_{i=0}^{n-1}$ and $\{S_{\alpha,\beta}^{-i}(0)\}_{i=0}^{n-1}$ coincide for any maps $f$.
\end{proof}

\begin{lemma}\label{lem3}
Let $\{I_f^i\}_{i=1}^{n}$ be partitions of $[0,1]$ determined by $\{f^{-i}(0)\}_{i=1}^{n-1}$ with point 0 and 1 such that each $I_f^i$ is closed interval and $\cup_i I_f^i=[0,1]$. Then, $f$ has a periodic point with period $n$, and these $n$ points belongs to an interior of each interval $I_f^i$ one each. 
\end{lemma}

\begin{proof}
By Lemma 3.6 in \cite{Keener}, it has already known that if the set $\{f^{-i}(0)\ |\ i=0,1,\cdots\}$ is finite and a number of the set is $n$, then $f$ has a periodic point with period $n$. Thus we shall show that each interval $I_f^i$, $i=1,\cdots,n$, possesses only one of points of a periodic point.

Assume that there are two points $p_1$ and $p_2$ of the periodic point in some interval $I_f^i$. Since the map has contracting property, all periodic points are stable and unstable fixed points or periodic points do not exist. Then there exists $k$ such that $f^{-k}(0)$ must be between $p_1$ and $p_2$. This is contradiction since the partition $\{I_f^i\}_i$ is made by the points $\{f^{-i}(0)\}_{i=1}^{n-1}$. 
 
\end{proof}

\subsection*{Proof of Theorem \ref{mainthm1}}

For convenience, we write $S_{\alpha,\beta}$ by $S$. We shall construct the homeomorphism $H$ such that $f\circ H=H\circ S$. First, let $\{I_f^i\}_{i=1}^{n}$ and $\{I_{S}^i\}_{i=1}^{n}$ be partitions of $[0,1]$ determined by $\{f^{-i}(0)\}_{i=1}^{n-1}$ and $\{S^{-i}(0)\}_{i=1}^{n-1}$ respectively. Each subintervals $I_f^i$ and $I_S^i$ have a periodic point in its interior by Lemma \ref{lem3}, we denote the periodic points by $p_f^i$ and $p_S^i$ for $i=1,\cdots,n$. Then set 
\begin{eqnarray}
h(p_S^i)=p_f^i, \ \ \ i=1,\cdots,n.
\end{eqnarray}

Next, by Lemma \ref{lem2}, since the orders of $S^{-i}(0)$ and $f^{-i}(0)$ are corresponding, it is enable to set 
\begin{eqnarray}
h(S^{-i}(0))=f^{-i}(0), \ \ \ i=0,\cdots,n-1.
\end{eqnarray}

Considering a orbits of $0$ by $f$, $\{f^i(0)\}_{i=0}^\infty$, the sequence $f^k(0),f^{n+k}(0),f^{2n+k}(0),\cdots$ goes to the periodic point $p_f^k$ from left side of the periodic point in $I_f^k$ by the map $f^n$. On the other hand, considering a orbits of $1$ by $f$, $\{f^i(1)\}_{i=0}^\infty$, the sequence $f^k(1),f^{n+k}(1),f^{2n+k}(1),\cdots$ goes to the periodic point $p_f^k$ from right side of the periodic point in $I_f^k$ by the map $f^n$. Then we correspond each orbits, that is, 
$$ 
h(S^i(0))=f^i(0),\ \ h(S^i(1))=f^i(1), \ \ \ i=1,2,\cdots.
$$

Finally, we define the function $h$ between these points $\{S^{i}(0)\}_{i=-(n-1)}^{\infty}$ and $\{p_S^i\}_{i=1}^{n}$. Consider a further partition for each $I_S^i$, $i=1,\cdots,n$, by the points $S^{i+mn}(0)$ and $S^{i+mn}(1)$ for $m=0,1,\cdots$. We first define

$$
h_{-(n-1)}:[S^{-(n-1)}(0),S^{1}(0)]\to[f^{-(n-1)}(0),f^{1}(0)]
$$
as an arbitrary homeomorphism. Next we define $h_{-(n-2)}:[S^{-(n-2)}(0),S^{2}(0)]\to[f^{-(n-2)}(0),f^{2}(0)]$ by $h_{-(n-2)}:=f\circ h_{-(n-1)}\circ S^{-1}$. Since $f$, $S^{-1}$ and $h_{-(n-1)}$ are all bijective, continuous and monotonic increasing on each domain, $h_{-(n-2)}$ becomes homeomorphism. and the following diagram holds.

\[
  \begin{CD}
     [S^{-(n-1)}(0),S^{1}(0)] @>{S}>> [S^{-(n-2)}(0),S^{2}(0)] \\
  @V{h_{-(n-1)}}VV    @V{h_{-(n-2)}}VV  \\
     [f^{-(n-1)}(0),f^{1}(0)]   @>{f}>>  [f^{-(n-2)}(0),f^{2}(0)]
  \end{CD}
\]

Inductively, we define $h_{-(n-m)}:[S^{-(n-m)}(0),S^{m}(0)]\to[f^{-(n-m)}(0),f^{m}(0)]$ as $h_{-(n-m)}:=f\circ h_{-(n-m)}\circ S$ for $m=3,4,5,\cdots$ satisfying the following diagram.

\[
  \begin{CD}
     [S^{-(n-m+1)}(0),S^{m-1}(0)] @>{S}>> [S^{-(n-m)}(0),S^{m}(0)] \\
  @V{h_{-(n-m+1)}}VV    @V{h_{-(n-m)}}VV  \\
     [f^{-(n-m+1)}(0),f^{m-1}(0)]   @>{f}>>  [f^{-(n-m)}(0),f^{m}(0)]
  \end{CD}
\]

Similarly, we construct the homeomorophism by using the image of $1$. We define
$$
g_{-(n-1)}:[S^{-(n-1)}(1),S^{1}(1)]\to[f^{-(n-1)}(1),f^{1}(1)]
$$
as an arbitrary homeomorphism. Next we define $g_{-(n-2)}:[S^{-(n-2)}(1),S^{2}(1)]\to[f^{-(n-2)}(1),f^{2}(1)]$ by $g_{-(n-2)}:=f\circ g_{-(n-1)}\circ S^{-1}$. Since $f$, $S^{-1}$ and $g_{-(n-1)}$ are all bijective, continuous and monotonic increasing on each domain, $g_{-(n-2)}$ becomes homeomorphism, and the following diagram holds.

\[
  \begin{CD}
     [S^{-(n-1)}(1),S^{1}(1)] @>{S}>> [S^{-(n-2)}(1),S^{2}(1)] \\
  @V{g_{-(n-1)}}VV    @V{g_{-(n-2)}}VV  \\
     [f^{-(n-1)}(1),f^{1}(1)]   @>{f}>>  [f^{-(n-2)}(1),f^{2}(1)]
  \end{CD}
\]

Inductively, we define $g_{-(n-m)}:[S^{-(n-m)}(1),S^{m}(1)]\to[f^{-(n-m)}(1),f^{m}(1)]$ as $g_{-(n-m)}:=f\circ g_{-(n-m+1)}\circ S^{-1}$ for $m=3,4,5,\cdots$ satisfying the following diagram.

\[
  \begin{CD}
     [S^{-(n-m+1)}(1),S^{m-1}(1)] @>{S}>> [S^{-(n-m)}(1),S^{m}(1)] \\
  @V{g_{-(n-m+1)}}VV    @V{g_{-(n-m)}}VV  \\
     [f^{-(n-m+1)}(1),f^{m-1}(1)]   @>{f}>>  [f^{-(n-m)}(1),f^{m}(1)]
  \end{CD}
\]

Finally, defining the map $h$ by
\begin{eqnarray}
H=\begin{cases}
h_i\ \ \  {\rm on}\ \ \ [S^{i-1}(0),S^{n+i-1}(0)]\\
g_i\ \ \  {\rm on}\ \ \ [S^{i-1}(1),S^{n+i-1}(1)]
\end{cases},\ \ i=-(n-1),-(n-2),\cdots,\nonumber
\end{eqnarray}
we obtained the homeomorophism $H$ satisfying $f\circ H=H\circ S$. 

\qed

\section{Periodic structure for family of nonlinear contracting maps}\label{S4}
In this section, we show that the family of nonlinear contracting maps constructed as follows possesses Farey structure for the parameter space.

Let $g:\mathbb{R}\to\mathbb{R}$ be a continuous monotonic increasing with $g(0)=0$ and $g(c^*)+1=c^*$. Assume that $g$ has contracting property, that is, there exists $\kappa<1$ such that 
$$
|g(x)-g(y)|\leq\kappa|x-y|\ \ \ \text{for}\ \ \ x,y\in(0,c^*),
$$
Define $h(x):=g(x)+1$. For $c\in (0,c_*)$, the map  
\begin{eqnarray}
T_c(x)=\begin{cases}
h(x)  & ({\rm if}\ x<c)\\
g(x)  & ({\rm if}\  x\geq c)
\end{cases}\label{Tc}
\end{eqnarray}
becomes a transformation on $[g(c),h(c)]$ which is a continuous monotonic increasing except with $x=c$. In figure \ref{fig2}, we draw the illustration of constructing the family of maps $\{T_c\}_{c\in (0,c^*)}$.
\begin{figure}[htpb]
\begin{center}
\includegraphics[bb=0 0 300 300, width=6cm]{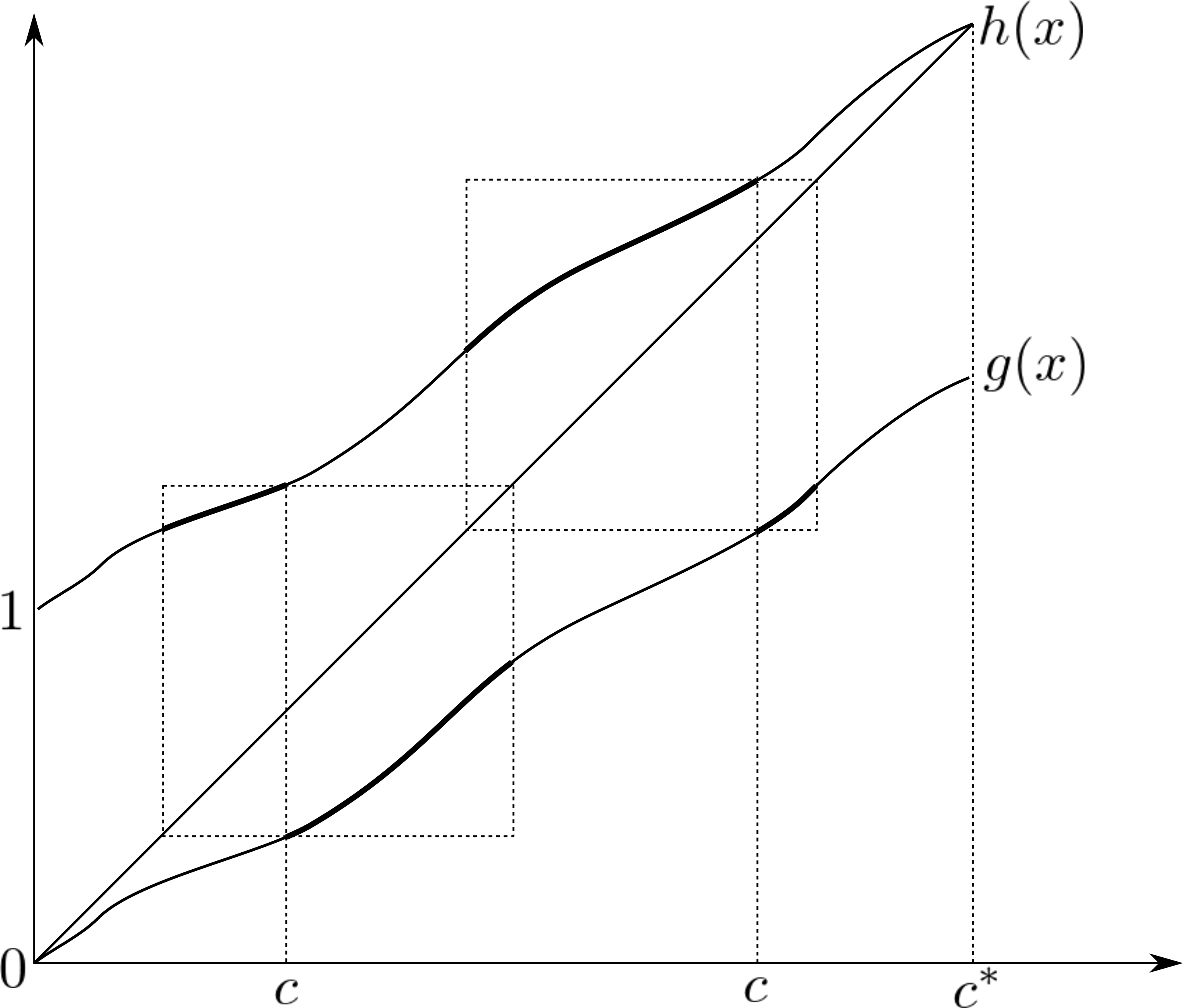}
\caption{The illustration of constructing the family of maps $\{T_c\}_{c\in(0,c^*)}$.}
\label{fig2}
\end{center}
\end{figure}

Next lemma implies that the transformation $T_c$ satisfies non-overlapping condition (A1). 

\begin{lemma}\label{nonoverlap}
The inequality $h\circ g(c)>g\circ h(c)$ holds for any $c\in(0,c^*)$.
\end{lemma}

\begin{proof}
Since $h(x)=g(x)+1$, the map $h$ clearly has contracting property. Thus, we have
$$
|g(c)-gh(c)|\leq\kappa|c-h(c)|,
$$
$$
|hg(c)-h(c)|\leq\kappa|g(c)-c|,
$$
which leads
\begin{eqnarray}
|g(c)-gh(c)| + |hg(c)-h(c)|&\leq&\kappa(|c-h(c)| + |g(c)-c|)\nonumber\\
&=&\kappa(h(c)-g(c)).\nonumber
\end{eqnarray}
Since $[g(c),gh(c)], [hg(c),h(c)]\subset [g(c),h(c)]$, we have $h\circ g(c)>g\circ h(c)$.
\end{proof}

Note that, in this section, we omit the character for a composition, $\circ$, that is, $gh(c)$ implies $g\circ h(c)$. Under these setting, the following theorem holds. 

\begin{theorem}\label{mainthm2}
 For the family of transformations $\{T_c\}_{c\in(0,c^*)}$ defined by \eqref{Tc}, there exist $c_{n,l}^L,c_{n,l}^R\in(0,c^*)$ for any $n\in\mathbb{N}_{\geq2}$ and $l\in Pr(n)$ such that
\begin{itemize}
\item[(i)] if $c\in (c_{n,l}^L,c_{n,l}^R)$, then $T_c$ conjugates $S_{\alpha,\beta}$ with $(\alpha,\beta)\in int(D_{n,l})$,
\item[(ii)] if $n'l-nl'=1$, then $c_{n,l}^R<c_{n+n',l+l'}^L<c_{n+n',l+l'}^R<c_{n',l'}^L$.
\end{itemize}
\end{theorem}

\begin{proof}

Since $T_c$ is clearly satisfied the assumptions (A1), (A2) and (A3) in Theorem \ref{mainthm1}, we have to show (A4) for the item (i). We use the special type of inductions based on the Farey series.

\vspace{\baselineskip}
{\bf (STEP 1)} For the case $(c_{2,1}^L,c_{2,1}^R)$.

Setting $C_{2,1}:=(0,c^*)$. Since $gh(C_{2,1})=(g(1),g(c^*))\subset C_{2,1}$ and $hg(C_{2,1})=(1,hg(c^*))\subset C_{2,1}$ and the compositions $gh$ and $hg$ are contraction mappings , there exist unique points $c'\in gh(C_{2,1})\subset C_{2,1}$ and $c''\in hg(C_{2,1})\subset C_{2,1}$ such that $c'=gh(c')$ and $c''=hg(c'')$ from the Banach's fixed point theorem. Moreover, we have that, for any $c\in C_{2,1}$,
\begin{eqnarray}
(\ast)_{2,1}: \begin{cases}
{\rm if}\ c<c',\ {\rm then}\ c<gh(c) \\
{\rm if}\ c>c',\ {\rm then}\ c>gh(c)
\end{cases}
{\rm and}\ \ \ 
\begin{cases}
{\rm if}\ c<c'',\ {\rm then}\ c<hg(c) \\
{\rm if}\ c>c'',\ {\rm then}\ c>hg(c)
\end{cases}.\nonumber
\end{eqnarray}

Assume that $c'\geq c''$, then there exists $\hat{c}\in (c'',c')$ such that $\hat{c}>hg(\hat{c})$ and $\hat{c}<gh(\hat{c})$, that is, $hg(\hat{c})<gh(\hat{c})$ holds, which is contradict to non-overlapping condition (Remark \ref{nonoverlap}). We then have $c'<c''$. By taking $c'$ and $c''$ as $c_{2,1}^L$ and $c_{2,1}^R$ respectively, we find that, for any $c\in (c_{2,1}^L, c_{2,1}^R)$, $gh(c)<c<hg(c)$ holds which means that $c$ is not in $[g(c),gh(c)]\cup[hg(c),h(c)]$.  Therefore, $T_c$ conjugate $S_{\alpha,\beta}$ with $(\alpha,\beta)\in int(D_{2,1})$ by Theorem \ref{mainthm1}.

\vspace{\baselineskip}
{\bf (STEP 2)} For the case $(c_{3,2}^L,c_{3,2}^R)$.

The idea is similar to (STEP 1). In the case, setting $C_{3,2}=(0,c_{2,1}^L)$, we have $g^{-1}(C_{3,2})=(0,h(c_{2,1}^L))$, $gh(C_{3,2})=(g(1),c_{2,1}^L)$ and $hg(C_{3,2})=(1,hg(c_{2,1}^L))$ so that $gh(C_{3,2})\subset g^{-1}(C_{3,2})$ and $hg(C_{3,2})\subset g^{-1}(C_{3,2})$ hold. From the Banach's fixed point theorem, there exists $c'\in ggh(C_{3,2})\subset C_{3,2}$ and $c''\in ghg(C_{3,2})\subset C_{3,2}$ such that $c'=ggh(c')$ and $c''=ghg(c'')$. Moreover, we have that, for any $c\in C_{3,2}$,
\begin{eqnarray}
(\ast)_{3,2}: \begin{cases}
{\rm if}\ c<c',\ {\rm then}\ c<ggh(c) \\
{\rm if}\ c>c',\ {\rm then}\ c>ggh(c)
\end{cases}
{\rm and}\ \ \ 
\begin{cases}
{\rm if}\ c<c'',\ {\rm then}\ c<ghg(c) \\
{\rm if}\ c>c'',\ {\rm then}\ c>ghg(c)
\end{cases}.\nonumber
\end{eqnarray}

Assume that $c'\geq c''$, then there exists $\hat{c}\in (c'',c')$ such that $\hat{c}>ghg(\hat{c})$ and $\hat{c}<ggh(\hat{c})$, that is, $hg(\hat{c})<gh(\hat{c})$ holds, which is contradict to non-overlapping condition (Remark \ref{nonoverlap}). We then have $c'<c''$. Take $c'$ and $c''$ as $c_{3,2}^L$ and $c_{3,2}^R$ respectively. Since $(c_{3,2}^L, c_{3,2}^R)\subset C_{2,1}$, we have $c\in [g(c),gh(c)]\cup[hg(c),h(c)]$ by $(\ast)_{2,1}$ for any $c\in (c_{3,2}^L, c_{3,2}^R)$. Moreover,  we find $ggh(c)<c<ghg(c)$, by $(\ast)_{3,2}$, that is, $gh(c)<g^{-1}(c)<hg(c)$ holds which means that $g^{-1}(c)$ is not in $[g(c),gh(c)]$. Therefore, $T_c$ conjugate $S_{\alpha,\beta}$ with $(\alpha,\beta)\in int(D_{3,2})$ by Theorem \ref{mainthm1}.

\vspace{\baselineskip}
{\bf (STEP 3)} For the case $(c_{n,n-1}^L,c_{n,n-1}^R)$.

Assume that we have already found a interval $(c^L_{n-1,n-2},c^R_{n-1,n-2})$ so that $c^L_{n-1,n-2}=g^{n-3}gh(c^L_{n-1,n-2})$. For the case $(c_{n,n-1}^L,c_{n,n-1}^R)$, setting $C_{n,n-1}=(0,c_{n-1,n-2}^L)$, we have $g^{-(n-2)}(C_{n,n-1})=(0,h(c_{n-1,n-2}^L))$, $gh(C_{n,n-1})=(g(1),gh(c_{n-1,n-2}^L)$ and $hg(C_{n,n-1})=(1,hg(c_{n-2,n-1}^L))$ so that $gh(C_{n,n-1})\subset g^{-(n-2)}(C_{n,n-1})$ and $hg(C)\subset g^{-(n-2)}(C)$ hold. From the Banach's fixed point theorem, there exists $c'\in g^{n-2}gh(C_{n,n-1})\subset C_{n,n-1}$ and $c''\in g^{n-2}hg(C_{n,n-1})\subset C_{n,n-1}$ such that $c'=g^{n-2}gh(c')$ and $c''=g^{n-2}hg(c'')$. Moreover, we have that, for any $c\in C_{n,n-1}$,
\begin{eqnarray}
(\ast)_{n,n-1}: \begin{cases}
{\rm if}\ c<c',\ {\rm then}\ c<g^{n-2}gh(c) \\
{\rm if}\ c>c',\ {\rm then}\ c>g^{n-2}gh(c) \\
{\rm if}\ c<c'',\ {\rm then}\ c<g^{n-2}hg(c) \\
{\rm if}\ c>c'',\ {\rm then}\ c>g^{n-2}hg(c)
\end{cases}.\nonumber
\end{eqnarray}

Assume that $c'\geq c''$, then there exists $\hat{c}\in (c'',c')$ such that $\hat{c}>g^{n-2}hg(\hat{c})$ and $\hat{c}<g^{n-2}gh(\hat{c})$, that is, $hg(\hat{c})<gh(\hat{c})$ holds, which is contradict to non-overlapping condition (Remark \ref{nonoverlap}). We then have $c'<c''$. Take $c'$ and $c''$ as $c_{n,n-1}^L$ and $c_{n,n-1}^R$ respectively. We know that 
$$
(c_{n,n-1}^L,c_{n,n-1}^R)\subset C_{n-1}\subset C_{n-1,n-2}\subset\cdots\subset C_{3,2}\subset C_{2,1}.
$$
Then, for any $c\in (c_{n,n-1}^L, c_{n,n-1}^R)$, by $(\ast)_{m,m-1}$, we have  
$$
g^{-(m-2)}(c)\in[g(c),gh(c)]\ \ {\rm for}\ \ m=2,3,\cdots,n-1,
$$
and by $(\ast)_{n,n-1}$,
$g^{n-2}gh(c)<c<g^{n-2}hg(c)$, that is, $gh(c)<g^{-(n-2)}(c)<hg(c)$ holds which means that $g^{-(n-1)}(c)$ is not in $[g(c),gh(c)]\cup[hg(c),h(c)]$. Therefore, $T_c$ conjugate $S_{\alpha,\beta}$ with $(\alpha,\beta)\in int(D_{3,2})$ by Theorem \ref{mainthm1}.
 
 \vspace{\baselineskip}
{\bf (STEP 4)} For the case $(c_{3,1}^L,c_{3,1}^R)$.

Under setting $C=(c_{2,1}^R,c^*)$, we can show the existences of $c_{3,1}^L$ and $c_{3,1}^R$ by the same way as (STEP 2) and substituting the roles of $g$ and $h$.

\vspace{\baselineskip}
{\bf (STEP 5)} For the case $(c_{n,1}^L,c_{n,1}^R)$.

Under setting $C=(c_{n-1,1}^R,c^*)$, we can show the existences of $c_{n,1}^L$ and $c_{n,1}^R$ by the same way as (STEP 3) and substituting the roles of $g$ and $h$.

\vspace{\baselineskip}
{\bf (STEP 6)} For the case $(c_{n,l}^L,c_{n,l}^R)$.

We show that for any $n\in\mathbb{N}_{\geq 2}$ and $l\in Pr(n)$, there exist $c_{n,l}^L, c_{n,l}^R$ such that 
\begin{itemize}
\item[(i)$_{(n,l)}$:] $c_{n,l}^L=v_1v_2\cdots v_{n-2}gh(c_{n,l}^L)$ and $c_{n,l}^R=v_1v_2\cdots v_{n-2}hg(c_{n,l}^R)$, 
\item[(ii)$_{(n,l)}$:] \begin{eqnarray}
\begin{cases}
{\rm if}\ c<c_{n,l}^L,\ {\rm then}\ c<v_1v_2\cdots v_{n-2}gh(c) \\
{\rm if}\ c>c_{n,l}^L,\ {\rm then}\ c>v_1v_2\cdots v_{n-2}gh(c)
\end{cases},\nonumber
\\ 
\begin{cases}
{\rm if}\ c<c_{n,l}^R,\ {\rm then}\ c<v_1v_2\cdots v_{n-2}hg(c) \\
{\rm if}\ c>c_{n,l}^R,\ {\rm then}\ c>v_1v_2\cdots v_{n-2}hg(c)
\end{cases},\nonumber
\end{eqnarray}
\item[(iii)$_{(n,l)}$:] $c_{n,l}^L<c_{n,l}^R$,
\item[(iv)$_{(n,l)}$:] for $c\in(c_{n,l}^L, c_{n,l}^R)$, $gh(c)<(v_1v_2\cdots v_{n-2})^{-1}(c)<hg(c)$,
\end{itemize}
where $\{v_i\}_{i=1}^{n-2}$ is defined by
\begin{eqnarray}\label{v}
v_i:=
\begin{cases}
h \ \ {\rm if} \ \ k_i=0\\
g \ \ {\rm if} \ \ k_i=1
\end{cases},
\end{eqnarray}
with $\{k_i\}$ is a rational characteristic sequence corresponding to $(n,l)$.

To show the above statement, we assume that the above holds for $(n,l)$ and $(n',l')$ with $n'l-nl'=1$. Then we will show the above statement for $(n+n',l+l')$. Let $\{v_i\}_{i=1}^{n-2}$ and $\{v'_i\}_{i=1}^{n'-2}$ be given by \eqref{v} with respect to $(n,l)$ and $(n',l')$ respectively. Set $C=(c_{n,l}^R,c_{n',l'}^L)$. From the assumption, we know
\begin{eqnarray}
c_{n,l}^R=v_1v_2\cdots v_{n-2}hg(c_{n,l}^R),\label{cR}\\
c_{n',l'}^L=v'_1v'_2\cdots v'_{n'-2}gh(c_{n',l'}^L).\label{cL}
\end{eqnarray}

First we show 
\begin{eqnarray}
gh(C)\subset(v'_1v'_2\cdots v'_{n'-2}ghv_1v_2\cdots v_{n-2})^{-1}(C)\label{gh}\\
hg(C)\subset(v'_1v'_2\cdots v'_{n'-2}ghv_1v_2\cdots v_{n-2})^{-1}(C)\label{hg}
\end{eqnarray}
Since $c_{n,l}^R<c_{n',l'}^L$, by (ii)$_{(n',l')}$, we have $c_{n,l}^R<v'_1v'_2\cdots v'_{n'-2}gh(c_{n,l}^R)$, that is,
\begin{eqnarray}
(v'_1\cdots v'_{n'-2})^{-1}(c_{n,l}^R) &<& gh(c_{n,l}^R)\nonumber\\
(v'_1\cdots v'_{n'-2})^{-1}g^{-1}h^{-1}(v_1\cdots v_{n-2})^{-1}(c_{n,l}^R) &<& gh(c_{n,l}^R)\ \ {\rm by}\ \ \eqref{cR}\nonumber\\
(v_1\cdots v_{n-2}hgv'_1\cdots v'_{n'-2})^{-1}(c_{n,l}^R) &<& gh(c_{n,l}^R)\nonumber\\
(v'_1\cdots v'_{n'-2}ghv_1\cdots v_{n-2})^{-1}(c_{n,l}^R) &<& gh(c_{n,l}^R)\label{11}
\end{eqnarray}
where we use the calculations pointed out in Remark \ref{palindrome}. 

Since $c_{n',l'}^L>c_{n,l}^R$, by (ii)$_{(n,l)}$, we have $c_{n',l'}^L>v_1v_2\cdots v_{n-2}hg(c_{n',l'}^L)$, that is,
\begin{eqnarray}
(v_1\cdots v_{n-2})^{-1}(c_{n',l'}^L) &>& hg(c_{n',l'}^L)\nonumber\\
(v_1\cdots v_{n-2})^{-1}(c_{n',l'}^L) &>& hgv'_1\cdots v'_{n'-2}gh(c_{n',l'}^L)\ \ {\rm by}\ \ \eqref{cL}\nonumber
\end{eqnarray}

\begin{eqnarray}
(v'_1\cdots v'_{n'-2})^{-1}g^{-1}h^{-1}(v_1\cdots v_{n-2})^{-1}(c_{n',l'}^L) &>& gh(c_{n',l'}^L)\nonumber\\
(v_1\cdots v_{n-2}ghv_1\cdots v_{n-2})^{-1}(c_{n',l'}^L) &>& gh(c_{n',l'}^L)\nonumber\\
(v'_1\cdots v'_{n'-2}ghv_1\cdots v_{n-2})^{-1}(c_{n',l'}^L) &>& gh(c_{n',l'}^L).\label{12}
\end{eqnarray}
where we use the calculations pointed out in Remark \ref{palindrome}. \eqref{11} and \eqref{12} imply \eqref{gh}. Similarly, since $c_{n,l}^R<c_{n',l'}^L$, by (ii)$_{(n',l')}$, we have $c_{n,l}^R<v'_1v'_2\cdots v'_{n'-2}gh(c_{n,l}^R)$, and by \eqref{cR} and the calculation in Remark \ref{palindrome}, 
\begin{eqnarray}
(v'_1v'_2\cdots v'_{n'-2}ghv_1v_2\cdots v_{n-2})^{-1}(c_{n,l}^R)<hg(c_{n,l}^R).\label{21}
\end{eqnarray}
Since $c_{n',l'}^L>c_{n,l}^R$, by (ii)$_{(n,l)}$, we have $c_{n',l'}^L>v_1v_2\cdots v_{n-2}hg(c_{n',l'}^L)$, and by \eqref{cR} and the calculation in Remark \ref{palindrome}, 
\begin{eqnarray}
(v'_1v'_2\cdots v'_{n'-2}ghv_1v_2\cdots v_{n-2})^{-1}(c_{n',l'}^L)>hg(c_{n',l'}^L)\ \ {\rm by}\ \ \eqref{cL}.\label{22}
\end{eqnarray}
\eqref{21} and \eqref{22} imply \eqref{hg}.

Thus, from the Banach's fixed point theorem, there exists $c'\in v_1v_2\cdots v_{n-2}hgv'_1v'_2\cdots v'_{n'-2}gh(C)\subset C$ and $c''\in v_1v_2\cdots v_{n-2}hgv'_1v'_2\cdots v'_{n'-2}hg(C)\subset C$ such that $c'=v_1v_2\cdots v_{n-2}hgv'_1v'_2\cdots v'_{n'-2}gh(c')$ and $c''=v_1v_2\cdots v_{n-2}hgv'_1v'_2\cdots v'_{n'-2}hg(c'')$. Moreover, we have that
\begin{eqnarray}
\begin{cases}
{\rm if}\ c<c',\ {\rm then}\ c<v_1v_2\cdots v_{n-2}hgv'_1v'_2\cdots v'_{n'-2}gh(c) \\
{\rm if}\ c>c',\ {\rm then}\ c>v_1v_2\cdots v_{n-2}hgv'_1v'_2\cdots v'_{n'-2}gh(c)
\end{cases},\nonumber
\\
\begin{cases}
{\rm if}\ c<c'',\ {\rm then}\ c<v_1v_2\cdots v_{n-2}hgv'_1v'_2\cdots v'_{n'-2}hg(c) \\
{\rm if}\ c>c'',\ {\rm then}\ c>v_1v_2\cdots v_{n-2}hgv'_1v'_2\cdots v'_{n'-2}hg(c)
\end{cases}.\nonumber
\end{eqnarray}
Assume that $c'\geq c''$, then there exists $\hat{c}\in (c'',c')$ such that $\hat{c}>v_1v_2\cdots v_{n-2}hgv'_1v'_2\cdots v'_{n'-2}hg(\hat{c})$ and $\hat{c}<v_1v_2\cdots v_{n-2}hgv'_1v'_2\cdots v'_{n'-2}gh(\hat{c})$, that is, $hg(\hat{c})<gh(\hat{c})$ holds, which is contradict to non-overlapping condition (Remark \ref{nonoverlap}). We then have $c'<c''$. By taking $c'$ and $c''$ as $c_{n+n',l+l'}^L$ and $c_{n+n',l+l'}^R$ respectively, we find that, for any $c\in (c_{n+n',l+l'}^L, c_{n+n',l+l'}^R)$, $v_1v_2\cdots v_{n-2}hgv'_1v'_2\cdots v'_{n'-2}gh(c)<c<v_1v_2\cdots v_{n-2}hgv'_1v'_2\cdots v'_{n'-2}hg(c)$, that is, $gh(c)<(v'_1v'_2\cdots v'_{n'-2}ghv_1v_2\cdots v_{n-2})^{-1}(c)<hg(c)$ holds which means that $(v'_1v'_2\cdots v'_{n'-2}ghv_1v_2\cdots v_{n-2})^{-1}(c)$ is not in $[g(c),gh(c])\cup[hg(c),h(c)]$. This completes the proof of the item (i). From the way to construct $c_{n,l}^L $and $c_{n,l}^R$, the item (ii) of the theorem clearly holds.  
\end{proof}

\begin{remark}
If $c=c_{n,l}^L$ (or $c_{n,l}^R$), we can prove similarly that $T_c$ conjugates $S_{\alpha,\beta}$ with $\beta=B^L_{n,l}(\alpha)$ (or $B^U_{n,l}(\alpha)$).
\end{remark}

\begin{example}
The family of maps $\{T_c\}$ defined by \eqref{Tc} for 
$g(x)=\alpha(\frac{1}{2}x+\frac{1}{4}\sin x)$. Since $g'(x)=\alpha(\frac{1}{2}+\frac{1}{4}\cos x)$ and $0<g'(x)<1$, the map $T_c$ always satisfies the assumptions for Theorem \ref{mainthm2}. The maximum value $c^*$ depends on $\alpha$ and the relation $\alpha=\frac{2(c^*-1)}{2c^*+\sin c^*}$ holds. In figure \ref{egfig1}, we display the periodic structure for the system. The number in each region implies the period.

\begin{figure}[htpb]
\begin{center}
\includegraphics[bb=0 0 700 350, width=10cm]{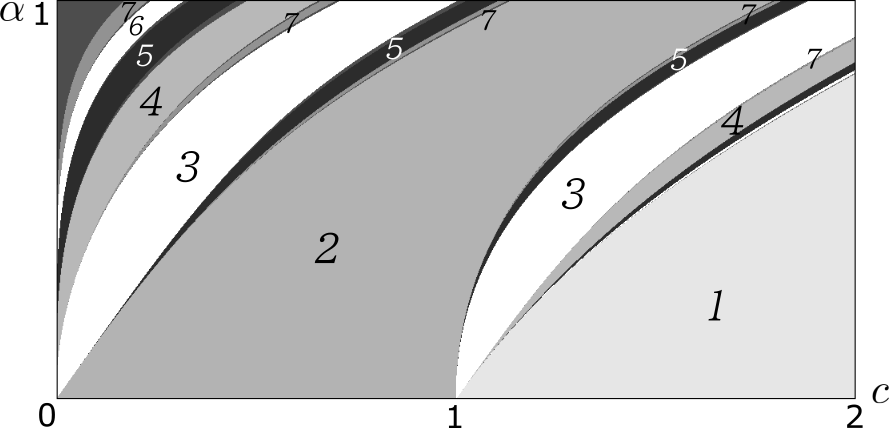}
\caption{The periodic structure for the family of maps $\{T_c\}$ defined by \eqref{Tc} together with $g(x)=\alpha(\frac{1}{2}x+\frac{1}{4}\sin x)$.}
\label{egfig1}
\end{center}
\end{figure}

\end{example}

In the end of this section, we give another type of result with Theorem \ref{mainthm2} in order to apply to transformations such as $T(x)=\alpha x^2+\beta$ (mod 1) or $T(x)=\alpha\sqrt{x}+\beta$ (mod 1), which are introduced in \cite{N1}.

Let $f:[0,1]\to[0,1]$ be a monotonically increasing continuous function such that $f(0)=0$, $f(1)<1$ and 
$|f(x)-f(y)|\leq\kappa|x-y|\ \ \ \text{for some $\kappa<1$ and $x,y\in[0,1]$}$. Define a family of transformation on $[0,1]$ by
\begin{eqnarray}
T_c(x)=
\begin{cases}
f(x)-f(c)+1 & (0\leq x < c)\\
f(x)-f(c)   & (c\leq x < 1)
\end{cases}.\label{Tc2}
\end{eqnarray}

\begin{theorem}\label{mainthm3}
For the family of transformations $\{T_c\}_{c\in(0,1)}$ defined by \eqref{Tc2}, there exist $c_{n,l}^L,c_{n,l}^R\in(0,1)$ for any $n\in\mathbb{N}_{\geq2}$ and $l\in Pr(n)$ such that
\begin{itemize}
\item[(i)] if $c\in (c_{n,l}^L,c_{n,l}^R)$, then $T_c$ conjugates $S_{\alpha,\beta}$ with $(\alpha,\beta)\in int(D_{n,l})$.
\item[(ii)] if $n'l-nl'=1$, then $c_{n,l}^R<c_{n+n',l+l'}^L<c_{n+n',l+l'}^R<c_{n',l'}^L$ .
\end{itemize}
\end{theorem}

\begin{proof}
For the case $(n,l)=(2,1)$, put $C_{2,1}=(0,1)$. Giving two function $h_c(x)=f(x)-f(c)-1$ and $g_c(x)=f(x)-f(c)$, let $H(c)=g_ch_c(c)$ and $G(c)=h_cg_c(c)$. In the proof of Theorem \ref{mainthm2}, substituting the role of $hg$ and $gh$ by $G$ and $H$, we have $G(C_{2,1})\subset C_{2,1}$ and $H(C_{2,1})\subset C_{2,1}$. Then we have $c'$ and $c''$ such that $c'=H(c')$ and $c''=G(c'')$ by the Banach fixed point theorem. By the same way, taking $c'$ and $c''$ as $c_{2,1}^L$ and $c_{2,1}^R$, we can show that for any $c\in(c_{2,1}^L,c_{2,1}^R)$, inequality $H(c)<c<G(c)$ holds which means that $c$ is not in $[g_c(c),g_ch_c(c)]\cup[h_cg_c(c),h(c)]$. Therefore, $T_c$ conjugates $S_{\alpha,\beta}$ with $(\alpha,\beta)\in int(D_{2,1})$ by Theorem \ref{mainthm1}.

For any case $(n,l)$, setting $H(c)=v_1^cv_2^c\cdots v_{n-2}^cg_ch_c(c)$ and $G(c)=v_1^cv_2^c\cdots v_{n-2}^ch_cg_c(c)$, where $v_i^c=h_c\ ({\rm if}\ k_i=0)$ or $g_c\ ({\rm if}\ k_i=1)$, we can give a similar proof with Theorem \ref{mainthm2}. 
\end{proof}

\begin{example}
Consider the family of maps $\{T_c\}$ defined by \eqref{Tc2} for $f(x)=\alpha x^2$. Then, putting $\beta=1-\alpha c^2$, the map becomes $T_c(x)=\alpha x^2+\beta$ (mod 1), which is a one of numerical example in \cite{N1}. Since $f'(x)=2\alpha x$, the map $T_c$ satisfies the assumptions of Theorem \ref{mainthm3} if $\alpha<1/2$. In figure \ref{egfig2}, we display the periodic structure for the system. The number in each region implies the period.

Although we can see the Farey structure for $\alpha\geq1/2$ from figure \ref{egfig2}, our theorem tells us the existence of the structure for $0<\alpha<1/2$. Because Theorem \ref{mainthm1} does not require the contracting property, and we can apply the Banach fixed point theorem in the proof of Theorem \ref{mainthm2} if the total derivative of $v_1\cdots v_n gh v_1'\cdots v'_{n'}$ is less than one, we expect Theorem \ref{mainthm2} can be hold under weakened condition.

\begin{figure}[htpb]
\begin{center}
\includegraphics[bb=0 0 400 320, width=6cm]{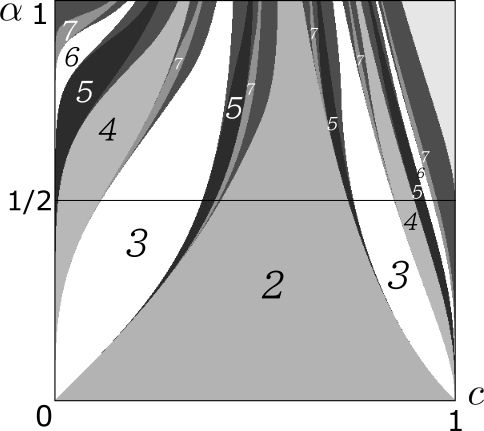}
\caption{The periodic structure for the family of maps $\{T_c\}$ defined by \eqref{Tc2} together with $f(x)=\alpha x^2$.}
\label{egfig2}
\end{center}
\end{figure}

\end{example}


\end{document}